\documentclass[11pt]{amsart}

\usepackage{amsmath}
\usepackage{a4wide}

\newcommand{\R}{{\mathbb R}}

\renewcommand{\div}{\mathrm{div}}

\newcommand{\Per}{\mathrm{P}}
\newcommand{\dia} {\mathrm{diam}}
\newcommand{\tr} {\mathrm{tr}}

\newtheorem{theorem}{Theorem}[section]

\newtheorem{lemma}[theorem]{Lemma}
\newtheorem{proposition}[theorem]{Proposition}

\theoremstyle{definition}

\newtheorem{remark}[theorem]{Remark}

\begin{document}
 
\title{Second-order  asymptotics  of the fractional perimeter as $s\to 1$}

\author{Annalisa Cesaroni}
\address{Dipartimento di Scienze Statistiche,
Universit\`{a} di Padova, Via Cesare Battisti 241/243, 35121 Padova, Italy}
\email{annalisa.cesaroni@unipd.it}
\author{Matteo Novaga}
\address{Dipartimento di Matematica, Universit\`{a} di Pisa, Largo Bruno Pontecorvo 5, 56127 Pisa, Italy}
\email{matteo.novaga@unipi.it}


\maketitle

\begin{abstract} In this note we provide a second-order  asymptotic  expansion of the fractional perimeter $\Per_s(E)$, as $s\to 1^-$, in terms of the local perimeter and of a higher order nonlocal functional.
\end{abstract} 

\tableofcontents

\section{Introduction} 
The fractional perimeter of a measurable set $E\subseteq\R^d$ is defined  as follows: 
\begin{equation}\label{pers1} 
\Per_s(E)=\int_E\int_{\R^d\setminus E} \frac{1}{|x-y|^{d+s}} dydx \qquad s\in (0,1).
\end{equation} 
After being first considered in the pivotal paper \cite{crs} (see also \cite{m} where the definition was first given), this functional  has inspired a variety of literature both in the community of pure mathematics, regarding for instance existence and regularity of  fractional minimal surfaces, and in view of  applications  to phase transition problems and to several models with long range interactions. We refer to \cite{v}, and references therein, for an introductory review on this subject.

The limits as $s\to 0^+$ or $s\to 1^-$  are critical,
in the sense that the fractional perimeter  \eqref{pers1} diverges to $+\infty$. Nevertheless, when appropriately
rescaled, such limits  give meaningful information on the set. 

The limit of the (rescaled) fractional perimeter when   $s\to 0^+$ has been considered in \cite{dfpv}, where the authors proved the pointwise convergence of $s \Per_s(E)$ 
to the volume functional $d\omega_{d} |E|$, for sets $E$ of finite perimeter, where $\omega_d$ is the volume of the ball of radius $1$ in $\R^d$.
The corresponding  second-order expansion has been recently considered  in \cite{dnp}. In particular it is shown that 
\begin{align*}  \Per_s(E)- \frac{d\omega_{d}}{s}|E|\stackrel{\Gamma}{\longrightarrow}  \int_{E}\int_{B_R(x)\setminus E}\frac{1}{|x-y|^d}dxdy
- \int_{E}\int_{E\setminus B_R(x)}\frac{1}{|x-y|^d}dxdy
-d\omega_d\log R |E|,\end{align*} with respect to the $L^1$-convergence  of the corresponding characteristic functions, where the limit functional  is independent of $R$,
and it is called the $0$-fractional perimeter.

The limit of $\Per_s(E)$ as $s\to 1^-$, in pointwise sense and in 
the sense of $\Gamma$-convergence,  has been studied  in \cite{adpm,cv}, where it is proved  that
\[
(1-s)\Per_s(E)\stackrel{\Gamma}{\longrightarrow} \omega_{d-1} \Per( E),
\] 
with respect to the  $L^1$-convergence. 

In this paper we are interested in the analysis of the next order expansion. In particular we will prove in Theorem \ref{gamma} that
\[ 
\frac{\omega_{d-1}\Per(E)}{1-s} - \Per_s(E)\stackrel{\Gamma}{\longrightarrow} \mathcal{H}(E)
\qquad \text{as $s\to 1^-$,}
\]   
with respect to the $L^1$-convergence, and 
the limit functional is defined as 
\begin{align}\label{limit}  \mathcal{H}(E):= &  \int_{\partial^* E}\int_{(E\Delta H^-(y))\cap B_{1}(y) }  \frac{|(y-x)\cdot \nu(y)|}{|x-y|^{d+1} } dxd\mathcal{H}^{d-1}(y)\\
& - \int_{\partial^* E} \int_{E\setminus B_1(y)} \frac{(y-x)\cdot \nu(y)}{|x-y|^{d+1} }  dxd\mathcal{H}^{d-1}(y)- \omega_{d-1}\Per(E)\nonumber \end{align} for sets $E$ with finite perimeter, 
and $\mathcal{H}(E)=+\infty$ otherwise.
Here we denote by $\partial^*E$ the reduced boundary of $E$, by $\nu(y)$ the outer normal to $E$ at $y\in \partial^* E$ and by $H^-(y)$ the hyperplane
 \[H^-(y):=\{x\in\R^d \ |\ (y-x)\cdot \nu(y)>0\}.\] 

We observe that, in dimension $d=2$,  the functional $\mathcal{H}(E)$ coincides with the $\Gamma$-limit as $\delta\to 0^+$ 
of the nonlocal energy 
\[2|\log \delta| \Per(E)-\int_{E}\int_{\R^2\setminus E} \frac{\chi_{(\delta, +\infty)}(|x-y|)}{|x-y|^3}dxdy,\] 
as recently proved by Muratov and Simon in \cite[Theorem 2.3]{ms}.
We also mention the recent work \cite{cnp}, where the authors establish the second-order expansion of appropriately rescaled  nonlocal functionals 
approximating Sobolev seminorms, recently considered by Bourgain, Brezis and Mironescu \cite{bbm}.

\medskip

As for the properties of  the limit functional $\mathcal{H}$, first of all we observe that it is coercive in the sense that it provides a control on the perimeter of the set, see Proposition \ref{coe}. Moreover it is bounded on $C^{1, \alpha}$ sets, for $\alpha>0$,  and on convex sets  $C$ such that  for some $s\in (0,1)$  the boundary integral $\int_{\partial^* C} H_s(C,x)d\mathcal{H}^{d-1}(x)$ is finite, where $H_s(C, x)$ is  the fractional mean curvature of $C$ at $x$, see Proposition \ref{remreg}.  In particular  when $E$ has  boundary of class $C^2$,  in Proposition \ref{repr} we show that the limit functional $\mathcal{H}( E)$ can be equivalently written as 
\begin{align*} \mathcal{H}( E)=&
\frac{1}{d-1}  \int_{\partial E} \int_{\partial E }  \frac{(\nu(x)-\nu(y))^2}{2|x-y|^{d-1}}d\mathcal{H}^{d-1}(x)d\mathcal{H}^{d-1}(y)
 -\frac{d\omega_{d-1}}{d-1}\Per(E) 
\\& +\frac{1}{d-1}  \int_{\partial  E} \int_{\partial E} \frac{1}{|x-y|^{d-1}}\left|\frac{(y-x)}{|y-x|}\cdot \nu(x)\right|^2((d-1)\log|x-y|-1) d\mathcal{H}^{d-1}(x)d\mathcal{H}^{d-1}(y) \\
&+ \int_{\partial E} \int_{\partial E } \frac{H(E,x)\nu(x)\cdot (y-x)}{|y-x|^{d-1}} \log|x-y| \, d\mathcal{H}^{d-1}(x)d\mathcal{H}^{d-1}(y)\end{align*}
where $H(E,x)$ denotes the (scalar) mean curvature at $x\in\partial E$, that is the sum of the principal curvatures divided by $d-1$. 
Notice that the first term in the expression above is the (squared) $L^2$-norm of a nonlocal second fundamental form of $\partial E$. We recall also that an analogous representation formula for the same functional in   dimension $d=2$, has been given in \cite{ms}. 

Some interesting issues about the limit functional remain open, for instance existence and rigidity (at least for small volumes) of  
minimizers of $\mathcal{H}$ among sets with fixed volume, see the discussion in Remark \ref{isorem}. 
 
\smallskip 
 
\paragraph{\bf Aknowledgements}
 The authors are members and were supported by the INDAM/GNAMPA.

 
\section{Second order asymptotics} 
We introduce the following functional on sets $E\subseteq \R^d$ of finite Lebesgue measure:
\begin{equation}\label{ps} \mathcal{P}_s(E)= \begin{cases} \frac{\omega_{d-1}}{1-s} \Per( E) -\Per_s(E) & \text{ if }\Per(E)<+\infty\\ 
+\infty &\text{otherwise.}\end{cases}
\end{equation} 
We now state the main result of the paper. 
 
 \begin{theorem}\label{gamma}
There holds   
\[
\mathcal{P}_s(E)\stackrel{\Gamma}{\longrightarrow} \mathcal{H}(E) \qquad \text{ as $s\to 1^-$,}  
\] 
with respect to the $L^1$-topology, where the functional $\mathcal{H}(E)$ is defined in \eqref{limit}.
\end{theorem}  

\begin{remark}\upshape Observe that $\mathcal{H}(E)$ can be also expressed as
\begin{align}\label{limit2} \mathcal{H}(E)= &  -\omega_{d-1}\Per(E)+ \int_{\partial^* E}\int_{(E\Delta H^-(y))\cap B_{1}(y) }  \frac{|(y-x)\cdot \nu(y)|}{|x-y|^{d+1} } dxd\mathcal{H}^{d-1}(y)\\
& +\int_{E} \int_{E\setminus B_1(x)} \frac{1}{|x-y|^{d+1} }  dydx-  \int_{E} \int_{\partial B_1(x)\cap E}  d\mathcal{H}^{d-1}(y)dx.\nonumber \end{align}

Indeed   by the divergence theorem and by the fact that $\div_y\left(\frac{y-x}{|y-x|^{d+1}}\right)=-\frac{1}{|y-x|^{d+1}}$  we get  
\begin{align}\label{perparti}
&-\int_{\partial^* E} \int_{E\setminus B_1(y)} \frac{(y-x)\cdot \nu(y)}{|x-y|^{d+1} } dxd\mathcal{H}^{d-1}(y) \\ \nonumber =&-\int_{E} \int_{\partial^* E\setminus B_1(x)} \frac{(y-x)\cdot \nu(y)}{|x-y|^{d+1} } d\mathcal{H}^{d-1}(y)dx
\\ \nonumber  =&\int_{E} \int_{E\setminus B_1(x)} \frac{1}{|x-y|^{d+1} }  dydx+\int_{E} \int_{\partial B_1(x)\cap E} \frac{(y-x)\cdot \frac{x-y}{|y-x|}}{|x-y|^{d+1} } d\mathcal{H}^{d-1}(y)dx 
\\    =&\int_{E} \int_{E\setminus B_1(x)} \frac{1}{|x-y|^{d+1} }  dydx-\int_{E} \int_{\partial B_1(x)\cap E}  d\mathcal{H}^{d-1}(y)dx.\nonumber 
\end{align} 
\end{remark}

First of all we recall some properties of the functional $\mathcal{P}_s$. 

\begin{proposition}[Coercivity and  lower semicontinuity] \label{compact} Let $ s\in (0,1)$. 
 If $E_n$ is a sequence of sets such that  $|E_n|\leq m$ for some $m>0$  and  $\mathcal{P}_s(E_n)\leq C$ for some $C>0$ independent of $n$, then  $\Per( E_n)\leq C'$ for some $C'$ depending on $C, s, d, m$. 
 
In particular, the sequence $E_n$ converges in $L^1_{\rm loc}$, up to a subsequence, to a limit set $E$ of finite perimeter, with  $|E|\leq m$.

Moreover,  the functional $\mathcal{P}_s$ is lower semicontinuous with respect to  the $L^1$-convergence. 
 \end{proposition}  
 
\begin{proof}[Proof of Proposition \ref{compact}] 
Let $E$ with $|E|\leq m$. By the interpolation inequality  proved in  \cite[Lemma 4.4]{blp}  we get 
\[\Per_s(E)\leq \frac{d\omega_d}{2^s s(1-s)}\Per(E)^s |E|^{1-s}\leq \frac{d\omega_d}{2^s s(1-s)}\Per(E)^s m^{1-s}. \] 
For a sequence $E_n$ as in the statement, this gives 
\begin{equation}\label{inter} C(1-s)\geq \omega_{d-1}\Per( E_n)- (1-s)\Per_{s}(E_n)\geq \omega_{d-1}\Per( E_n) - \frac{d\omega_d}{2^s s }\Per(  E_n)^s m^{1-s}. 
\end{equation}
From this we conclude that necessarily  $\Per( E_n)\leq C'$, where $C'$ is a constant which depends on $C, s, d, m$. As a consequence, 
by the local compactness in $L^1$ of sets of finite perimeter (see \cite{maggibook}) we obtain the local convergence of $E_n$,
up to a subsequence, to a limit set $E$ of finite perimeter.

Now, assume that $E_n\to E$ in $L^1$ and that  $\frac{c}{1- s}\Per( E_n)- \Per_{s}(E_n)\leq C$. By the previous argument, we get that $\Per(E_n)\leq C'$, where $C'$ is a constant which depends on $C, s, d, |E|$. By the compact embedding of $BV$ in  $H^{s/2}$, see \cite{hit, m}, we get that  $\lim_n \Per_s(E_n)=\Per_s(E)$, up to passing to a suitable subsequence.  This, along with the  lower semicontinuity of the perimeter with respect to local convergence in $L^1$ (see \cite{maggibook}) gives the conclusion. 
\end{proof}

The proof  of Theorem \ref{gamma} is based on some preliminary results. First of all we compute  the pointwise limit, then we  show that the functional $s\mathcal{P}_s(E)$ is given by the sum of the  functional $\mathcal{F}_s(E)$, defined in \eqref{fis},  which is lower semicontinuous and monotone increasing in $s$, and of a continuous functional. This will permit to show that the pointwise limit coincides with the $\Gamma$-limit. 

\begin{proposition}[Pointwise limit]  \label{pointwise}   
Let   $E\subseteq\R^d $ be a measurable set  such that  $|E|<+\infty$ and $\Per(E)<+\infty$.  Then 
\[ \lim_{s\to 1^-} \left[  \frac{\omega_{d-1}}{1-s} \Per(E)- \Per_s(E)\right]=\begin{cases} \mathcal{H}(E) & \text{ if $ \int_{\partial^* E}\int_{(E\Delta H^-(y))\cap B_{1}(y) }  \frac{|(y-x)\cdot \nu(y)|}{|x-y|^{d+1} } dxd\mathcal{H}^{d-1}(y)<+\infty$ }\\ +\infty &\text{otherwise.} \end{cases}\] 
where $\mathcal{H}(E)$ is defined in \eqref{limit} and  $H^-(y):=\{x\in\R^d \ |\ (y-x)\cdot \nu(y)>0\}$.
\end{proposition} 

\begin{proof} 
We can write $\Per_s(E)$  as a boundary integral observing that for all $0<s< 1$
  \begin{equation}\label{div}\div_y\left(\frac{y-x}{|y-x|^{d+s}}\right)=-s\frac{1}{|y-x|^{d+s}}.    \end{equation}
So, by  the divergence theorem, \eqref{pers1} reads \begin{align}\label{pers2} 
\Per_s(E) =& \frac{1}{s }\int_{\partial^* E}\int_{E} \frac{(y-x)\cdot \nu(y)}{|x-y|^{d+s} } dxd\mathcal{H}^{d-1}(y)\\\nonumber 
=& \frac{1}{s }\int_{\partial^*  E}\int_{E\cap B_1(y)} \frac{(y-x)\cdot \nu(y)}{|x-y|^{d+s} }dxd\mathcal{H}^{d-1}(y)\\ &+ \frac{1}{s }\int_{\partial^* E}\int_{E\setminus B_1(y)} \frac{(y-x)\cdot \nu(y)}{|x-y|^{d+s} } dxd\mathcal{H}^{d-1}(y)\nonumber
\end{align}
where $\nu(y)$ is the outer normal at $\partial^* E$ in $y$ and $R>0$.

We fix now  $y\in \partial^* E$     and we observe that, since $H^-(y):=\{x\in\R^d \ |\ (y-x)\cdot \nu(y)>0\}$, 
\begin{align} \label{c1} & \int_{E\cap B_1(y)} \frac{(y-x)\cdot \nu(y)}{|x-y|^{d+s} } dx \\ \nonumber =& \int_{H^-(y)\cap B_{1}(y) }   \frac{(y-x)\cdot \nu(y)}{|x-y|^{d+s} } dx+ 
 \int_{(E\setminus H^-(y)) \cap B_{1}(y) }   \frac{(y-x)\cdot \nu(y)}{|x-y|^{d+s} } dx\\ \nonumber &- \int_{(H^-(y)\setminus E) \cap B_{1}(y) }   \frac{(y-x)\cdot \nu(y)}{|x-y|^{d+s} } dx\\ \nonumber 
 =& \int_{H^-(y)\cap B_{1}(y) }   \frac{(y-x)\cdot \nu(y)}{|x-y|^{d+s} } dx-  \int_{(E\Delta H^-(y)) \cap B_{1}(y) }   \frac{|(y-x)\cdot \nu(y)|}{|x-y|^{d+s} } dx.   \end{align}
Now  we compute, denoting by $B'_1$ the ball in $\R^{d-1}$ with radius $1$ (and center $0$),  
\begin{align}  \label{c2} \int_{H^-(y)\cap B_{1}(y) }   \frac{(y-x)\cdot \nu(y)}{|x-y|^{d+s} } dx&  
 = \int_{\{x_d\geq 0\}\cap B_{1}} \frac{x_d}{|x|^{d+s}}dx\\ \nonumber &= \int_{B_{1}'} \int_0^{\sqrt{1-|x'|^2}} \frac{x_d}{(x_d^2+|x'|^2)^{(d+s)/2}}dx_d \\\nonumber &=\int_{B'_{1}} \frac{1}{2-d-s} (1-|x'|^{2-d-s})dx' = \omega_{d-1}\frac{1}{1-s}. \end{align}
If we substitute \eqref{c2}  in \eqref{c1} we get
\begin{equation}\label{c4} 
\int_{E\cap B_1(y)} \frac{(y-x)\cdot \nu(y)}{|x-y|^{d+s} } dx=  \frac{\omega_{d-1}}{1-s}-  \int_{(E\Delta H^-(y) )\cap B_{1}(y) }   \frac{|(y-x)\cdot \nu(y)|}{|x-y|^{d+s} } dx. 
\end{equation}   
By \eqref{pers2} and  \eqref{c4}  we obtain
\begin{align}\label{ultimo}  \omega_{d-1}\frac{ \Per(E)}{(1-s)} -\Per_s(E)= &  \omega_{d-1}\frac{ \Per(E)}{(1-s)}
-\omega_{d-1}\frac{\Per(E)}{s(1-s)}\\ &+ \frac{1}{s}\int_{\partial^* E} \int_{(E\Delta H^-(y)) \cap B_{1}(y) }   \frac{|(y-x)\cdot \nu(y)|}{|x-y|^{d+s} }dxd\mathcal{H}^{d-1}(y) \nonumber \\&-\frac{1}{s}\int_{\partial^* E} \ \int_{E\setminus B_1(y)} \frac{(y-x)\cdot \nu(y)}{|x-y|^{d+s} }dxd\mathcal{H}^{d-1}(y).
 \nonumber \end{align} 
Now we observe that, by Lebesgue's dominated convergence theorem, there holds 
\begin{equation}\label{fuori}  \lim_{s\to 1^-}\frac{1}{s} \int_{\partial^* E} \int_{E\setminus B_1(y)} \frac{(y-x)\cdot \nu(y)}{|x-y|^{d+s} } dxd\mathcal{H}^{d-1}(y)=   \int_{\partial^* E} \int_{E\setminus B_1(y)} \frac{(y-x)\cdot \nu(y)}{|x-y|^{d+1} } dxd\mathcal{H}^{d-1}(y) . \end{equation}
Moreover, by the monotone convergence theorem, 
\begin{align}\label{dentro} \lim_{s\to 1^-}  \int_{(E\Delta H^-(y))\cap B_1(y)  }  \frac{|(y-x)\cdot \nu(y)|}{|x-y|^{d+s} } dx =  \int_{(E\Delta H^-(y))\cap B_{1}(y) }  \frac{|(y-x)\cdot \nu(y)|}{|x-y|^{d+1} } dx \end{align} if  $ \frac{|(y-x)\cdot \nu(y)|}{|x-y|^{d+1} }\in L^1((E\Delta H^-(y))\cap B_1(y))$ and $ \lim_{s\to 1^-}  \int_{(E\Delta H^-(y))\cap B_1(y)  }  \frac{|(y-x)\cdot \nu(y)|}{|x-y|^{d+s} } dx= +\infty $ otherwise. 
The conclusion then follows from  \eqref{ultimo}, \eqref{fuori},    \eqref{dentro} sending $s\to 1^-$. 
\end{proof} 

\begin{lemma}\label{lemmamono} 
For  $s\in (0,1)$ and $E\subseteq\R^d$ of finite measure, we define the functional 
 \begin{align}\label{fis}   \mathcal{F}_s(E):= \begin{cases}  s\left[\frac{\omega_{d-1}}{1-s} \Per(E) -\Per_s(E)-\int_{E} \int_{E\setminus B_1(x)} \frac{1}{|x-y|^{d+s} }  dydx\right]& \text{ if $\Per(E)<+\infty$}\\ +\infty& \text{otherwise}.\end{cases}\end{align}
 
Then the following holds:
\begin{enumerate} 
\item The map  $s\mapsto \mathcal{F}_s(E)$ is monotone increasing as $s\to 1^-$. Moreover,  for every    $E$ of finite perimeter  
\begin{align*}\lim_{s\to 1^-}\mathcal{F}_s(E)=& -\omega_{d-1}\Per(E) + \int_{\partial^* E}\int_{(E\Delta H^-(y))\cap B_{1}(y) }  \frac{|(y-x)\cdot \nu(y)|}{|x-y|^{d+1} } dxd\mathcal{H}^{d-1}(y)\\  & - \int_{E} \int_{  \partial B_1(x)\cap E}  d\mathcal{H}^{d-1}(y)dx.\end{align*}
\item  For every family of  sets $E_s$  such that  $\mathcal{F}_s(E_s)\leq C$, for some $C>0$ independent of $s$, and $E_s\to E$  in $L^1$, there holds
\begin{align*}\liminf_{s\to1} \mathcal{F}_s(E_s) \geq & - \omega_{d-1}\Per(E)\\  &+ \int_{\partial^*E}\int_{(E\Delta H^-(y))\cap B_{1}(y) }  \frac{|(y-x)\cdot \nu(y)|}{|x-y|^{d+1} } dxd\mathcal{H}^{d-1}(y)- \int_{E} \int_{  \partial B_1(x)\cap E}  d\mathcal{H}^{d-1}(y)dx.\end{align*} \end{enumerate}
\end{lemma} 

\begin{proof} 
\begin{enumerate}\item 
Arguing as in  \eqref{perparti}   and using \eqref{div},  we get 
\begin{align*} \mathcal{F}_s(E)= & s\Big[\frac{\omega_{d-1}}{1-s} \Per(E) -\Per_s(E)\\&+\frac{1}{s}\int_{\partial^* E} \int_{E\setminus B_1(y)} \frac{(y-x)\cdot \nu(y)}{|x-y|^{d+s} } dxd\mathcal{H}^{d-1}(y) - \frac{1}{s}\int_{E} \int_{  \partial B_1(x)\cap E}  d\mathcal{H}^{d-1}(y)dx\Big].\end{align*}
Therefore from \eqref{pers2},  and \eqref{c4},   we get for $0<\bar s<s<1$
\begin{align*} &\frac{\mathcal{F}_s(E)+ \int_{E} \int_{  \partial B_1(x)\cap E}  d\mathcal{H}^{d-1}(y)dx}{s}\\ =&\omega_{d-1}\frac{ \Per(E)}{(1-s)}-\Per_s(E) +\frac{1}{s} \int_{\partial^* E} \int_{E\setminus B_1(y)} \frac{(y-x)\cdot \nu(y)}{|x-y|^{d+s} } dxd\mathcal{H}^{d-1}(y)\\ =& \omega_{d-1}\frac{ \Per(E)}{(1-s)}-
\frac{1}{s}  \int_{\partial^* E}\int_{E\cap B_1(y)} \frac{(y-x)\cdot \nu(y)}{|x-y|^{d+s} } dxd\mathcal{H}^{d-1}(y) \nonumber \\  =& -\frac{\omega_{d-1}}{s}  \Per(E)  +\frac{1}{s}\int_{\partial^* E}  \int_{(E\Delta H^-(y)) \cap B_{1}(y) }   \frac{|(y-x)\cdot \nu(y)|}{|x-y|^{d+s} } dxd\mathcal{H}^{d-1}(y)
\\ >& -\frac{\omega_{d-1}}{s}  \Per(E)  +\frac{1}{s}\int_{\partial^* E}  \int_{(E\Delta H^-(y) ) \cap B_{1}(y) }   \frac{|(y-x)\cdot \nu(y)|}{|x-y|^{d+\bar s} } dxd\mathcal{H}^{d-1}(y)\\
=& \frac{\mathcal{F}_{\bar s}(E)+ \int_{E} \int_{  \partial B_1(x)\cap E}  d\mathcal{H}^{d-1}(y)dx}{s}, 
\end{align*} 
which gives the desired monotonicity. 

Now we observe that  by the dominated convergence for every $E$ with $|E|<+\infty$ and $\Per(E)<+\infty$,  
\begin{align*} &\lim_{s\to 1}\frac{1}{s}\int_{\partial^* E} \int_{E\setminus B_1(y)} \frac{(y-x)\cdot \nu(y)}{|x-y|^{d+s} } dxd\mathcal{H}^{d-1}(y) - \frac{1}{s}\int_{E} \int_{  \partial B_1(x)\cap E}  d\mathcal{H}^{d-1}(y)dx\\ &= \int_{\partial^* E} \int_{E\setminus B_1(y)} \frac{(y-x)\cdot \nu(y)}{|x-y|^{d+1} } dxd\mathcal{H}^{d-1}(y) -\int_{E} \int_{  \partial B_1(x)\cap E}  d\mathcal{H}^{d-1}(y)dx \end{align*}
So, we conclude by Proposition \ref{pointwise}. 

\item We fix a family of sets $E_s$ such that $\mathcal{F}_s(E_s)\leq C$ and $E_s\to E$ in $L^1$ as $s\to 1^-$.  
Fix $\bar s<1$ and observe that by the monotonicity property proved in item (i), we get 
\begin{align*}&\liminf_{s\to 1}\mathcal{F}_s(E_s)\geq \liminf_{s\to 1}\mathcal{F}_{\bar s}(E_s)\\ \geq& \liminf_{s\to 1}\bar s\left[\frac{\omega_{d-1}}{1-\bar s} \Per( E_s) -\Per_{\bar s}(E_s)\right]-\lim_{s\to 1}\bar s \int_{E_s} \int_{E_s\setminus B_1(x)} \frac{1}{|x-y|^{d+\bar s} }  dydx\\
 \geq& \bar s\left[\frac{\omega_{d-1}}{1-\bar s} \Per(E) -\Per_{\bar s}(E)\right]- \bar s \int_{E} \int_{E\setminus B_1(x)} \frac{1}{|x-y|^{d+\bar s} }  dydy= \mathcal{F}_{\bar s}(E)
 \end{align*}
where we used for the first limit  the lower semicontinuity proved in Proposition \ref{compact},  and the dominated convergence theorem  for the second limit. 

We conclude by item (i), observing that $\mathcal{F}_{\bar s}(E)<C$, and   sending $\bar s\to 1^-$. 
 \end{enumerate} 
 \end{proof} 
 
We are now ready to prove our main result. 
 
\begin{proof}[Proof of Theorem \ref{gamma}] 
We start with the $\Gamma$-liminf inequality. Let $E_s$ be a sequence of sets such that  
$E_s\to E$ in $L^1$. 
We will prove that 
\[\liminf_{s\to 1} s\left[\frac{\omega_{d-1}}{1-s} \Per(  E_s) -\Per_s(E_s)\right] \geq \mathcal{H}(E),\] 
which will give immediately  the conclusion.  Recalling the definition of $\mathcal{F}_s(E)$ given in \eqref{fis}, we have that 
\[\liminf_{s\to 1} s\left[\frac{\omega_{d-1}}{1-s} \Per( E_s) -\Per_s(E_s)\right] \geq \liminf_{s\to 1}\mathcal{F}_s(E_s)+\liminf_{s\to 1} s\int_{E_s}\int_{E_s\setminus B_1(x)} \frac{1}{|x-y|^{d+s}}dydx.\] 
By Proposition \ref{lemmamono}, item (ii) and by Fatou lemma, we get 
\begin{align*}& \liminf_{s\to 1} s\left[\frac{\omega_{d-1}}{1-s} \Per( E_s) -\Per_s(E_s)\right] \geq -\omega_{d-1}\Per(E)\\  +& \int_{\partial^* E}\int_{(E\Delta H^-(y))\cap B_{1}(y) }  \frac{|(y-x)\cdot \nu(y)|}{|x-y|^{d+1} } dxd\mathcal{H}^{d-1}(y)- \int_{E} \int_{  \partial B_1(x)\cap E}  d\mathcal{H}^{d-1}(y)dx\\ 
+& \int_{E}\int_{E\setminus B_1(x)} \frac{1}{|x-y|^{d+1}}dydx=\mathcal{H}(E)\end{align*}
where the last equality comes from \eqref{perparti}.

The $\Gamma$-limsup  is a consequence of  the pointwise limit in Proposition \ref{pointwise}.  \end{proof}

We conclude this section with the equi-coercivity of the family of functionals $\mathcal{P}_s$, which is   a consequence of the monotonicity property of $\mathcal{F}_s$ obtained in Lemma \ref{lemmamono}. 

\begin{proposition}[Equi-coercivity]\label{equicoe}  Let $s_n$ be a sequence of positive  numbers  with $s_n\to 1^-$,  let $m,\,C\in\R$ with $m>0$,
and let  $E_n$ be a sequence of  measurable sets such that $|E_n|\leq m$  and  $\mathcal{P}_{s_n}(E_n)\leq C$ for all $n\in\mathbb N$.

Then  $\Per( E_n)\leq C'$ for some $C'>0$ depending on $C, d, m$, and  the sequence $E_n$ converges in $L^1_{\rm loc}$, up to a subsequence, to a limit set $E$ of finite perimeter, with  $|E|\leq m$.
\end{proposition}  

\begin{proof} 
Reasoning as in Proposition \ref{compact}, we get that $E_n$ has finite perimeter, for every $n\in \mathbb N$. 
Recalling \eqref{fis}, we get that 
\[ |C|\geq  s_n \mathcal{P}_{s_n}(E_n)=\mathcal{F}_{s_n}(E_n)+s_n\int_{E_n}\int_{E_n\setminus B_1(x)}\frac{1}{|x-y|^{d+s_n}}dydx\geq \mathcal{F}_{s_n}(E_n).\] 
We fix now $\bar n$ such that $s_{\bar n}>\frac{1}{2}$ and we claim that there exists $C'$, depending on $m, d$ but independent of $n$, such that $\Per( E_n)\leq C'$ for every $n\geq \bar n$.  If the claim is true, then it is immediate to conclude that eventually enlarging $C'$, $\Per( E_n)\leq C'$ for every $n$. 

For every $n\geq \bar n$, we use  the monotonicity  of the map $s\mapsto \mathcal{F}_s(E_n)$ proved in Lemma \ref{lemmamono},  and the fact that $|E_n|\leq m$, to obtain that 
\begin{align*} |C|&\geq \mathcal{F}_{s_n}(E_n)\geq \mathcal{F}_{s_{\bar n}}(E_n)= s_{\bar n}\mathcal{P}_{s_{\bar n}}(E_n)-s_{\bar n}\int_{E_n}\int_{E_n\setminus B_1(x)}\frac{1}{|x-y|^{d+s_{\bar n}}}dydx\\ & \geq  s_{\bar n}\mathcal{P}_{s_{\bar n}}(E_n)-s_{\bar n}\int_{E_n}\int_{E_n\setminus B_1(x)}dydx\geq s_{\bar n}\mathcal{P}_{s_{\bar n}}(E_n)-s_{\bar n}|E_n|^2\geq s_{\bar n}\mathcal{P}_{s_{\bar n}}(E_n)-s_{\bar n}m^2.
\end{align*}  
This implies in particular that $\mathcal{P}_{s_{\bar n}}(E_n)\leq \frac{|C|}{s_{\bar n}}+ m^2\leq 2|C|+m^2$, and  we conclude by Proposition \ref{compact}. 
\end{proof} 

\begin{remark}[Isoperimetric problems] \upshape \label{isorem}
Let us consider the following isoperimetric-type problem
for the functionals $\mathcal{P}_s$ and $\mathcal{H}$:
\begin{eqnarray}\label{iso} \min_{|E|=m}\mathcal{P}_s(E)\\\label{iso2} 
   \min_{|E|=m}\mathcal{H}(E),\end{eqnarray} 
   where $m>0$ is a fixed constant.
Observe that  $\widetilde E$ is a minimizer of \eqref{iso} if and only if the rescaled set $m^{-\frac{1}{d}}\widetilde E$ is a minimizer of 
\begin{equation*}\label{isoresc} 
\min_{|E|=1} \frac{\omega_{d-1}}{1-s} \Per(E)-m^{\frac{1-s}{d}} \Per_s(E).
 \end{equation*} 
Note in particular that the functional $\mathcal{P}_s$  is  given by the sum of an attractive term, which is the perimeter functional, 
and a repulsive term given by the fractional perimeter with a negative sign.

In general we cannot expect existence of solutions to these problems for every value of $m$.
However, from \cite[Thm 1.1, Thm 1.2]{dnrv} it follows that there exist $0<m_2(s)\leq m_1(s)$ such that, for all $m<m_1(s)$, 
Problem   \eqref{iso} admits a solution and moreover, if $m<m_2(s)$, the unique solution (uo to translations) is the ball of volume $m$.  
Actually, the  bounds $m_1(s), m_2(s)$ tend to $0$ as $s\to 1^-$, hence these results cannot be extended directly
to Problem \eqref{iso2}. 

A weaker notion of solution, introduced in \cite{kmn}, are the so-called generalized minimizers, that is, 
minimizers of the functional $\sum_i \mathcal{P}_s(E_i)$ (resp. of $\sum_i \mathcal{H}(E_i)$), among sequences of sets 
$(E_i)_{i}$  such that $|E_i|>0$ and $P(E_i)<+\infty$ for finitely many $i$'s, and $\sum_i |E_i|=m$. Note that, 
if  $E_n$ is a minimizing sequence for \eqref{iso} or \eqref{iso2},  
by reasoning as in Proposition \ref{equicoe}, we get that  there exists a constant $C=C(m)>0$ 
such that $\Per(E_n)\leq C$ for every $n$.  
Then, as it is proved in \cite[Proposition 2.1]{fl},  there exists $C'=C'(m)>0$, depending on $C$ and $m$,   
such that $\sup_x|E_n\cap B_1(x)|\geq C'$. Using these facts, reasoning as in \cite{kmn}, it is possible to show existence 
of generalized minimizers both for  \eqref{iso} and \eqref{iso2}, for every value of $m>0$. 
\end{remark}

\section{Properties of the limit functional}
In this section we analyze the main properties of  the limit functional $\mathcal{H}$. Note that, since it is obtained as a $\Gamma$-limit, it is naturally lower semicontinuous with respect to $L^1$ convergence.  

First of all we observe that by the representation of $\mathcal{H}$ in \eqref{limit2}, for every $E$ with finite perimeter there holds
\begin{align} \label{h}
 -\omega_{d-1}\Per(E) -d\omega_d|E|  \leq   \mathcal{H}(E) &  \leq  \int_{\partial^* E}\int_{(E\Delta H^-(y))\cap B_{1}(y) }  \frac{|(y-x)\cdot \nu(y)|}{|x-y|^{d+1} } dxd\mathcal{H}^{d-1}(y)+  d\omega_d|E|
 \\
&  \leq  \int_{\partial^* E}\int_{(E\Delta H^-(y))\cap B_{1}(y) }  \frac{1}{|x-y|^{d} } dxd\mathcal{H}^{d-1}(y)+  d\omega_d|E|. \nonumber 
\end{align}

We start with a compactness property in $L^1$ for sublevel sets of $\mathcal{H}$, which follows from  a lower bound on $\mathcal{H}$
in terms of the perimeter.

\begin{proposition} \label{coe} Let $E\subseteq \R^d$ be such that $\mathcal{H}(E)\leq C$. Then there exists a constant $C'$ depending on $C,|E|,d$ such that  $\Per(E)\leq C'$. 

In particular, if $E_n$ is a sequence of sets such that $\mathcal{H}(E_n)\leq C$, then there exists a limit set
$E$ of finite perimeter such that $\mathcal{H}(E)\leq C$  
and $E_n\to E$ in $L^1_{\rm loc}$ as $n\to +\infty$, up to a subsequence.
\end{proposition} 

\begin{proof} 
By Lemma \ref{lemmamono}, for $s\in (0,1)$ there holds
\[ \mathcal{F}_s(E)\leq \mathcal{H}(E)-\int_{E}\int_{E_n\setminus B_1(y)} \frac{1}{|x-y|^{d+1}}dxdy\leq \mathcal{H}(E)\leq C.
\]
The estimate on $P(E_n)$ then follows  by Proposition \ref{equicoe}. 
 
The second statement is a direct consequence of the lower semicontinuity of $\mathcal{H}$,
and of the local compactness in $L^1$ of sets of finite perimeter.
\end{proof} 

We point out  the following rescaling property of the functional $\mathcal{H}$, the will allow us to consider only sets with diameter less than $1$. 

\begin{proposition}
For every $\lambda>0$ there holds  \begin{equation}\label{resc}\mathcal{H}(\lambda E)=\lambda^{d-1}\mathcal{H}(E)-\omega_{d-1} \lambda^{d-1}\log \lambda \Per(E). \end{equation} 
\end{proposition} 

\begin{proof}
We observe that for every $R> 0$, with the same computation as in \eqref{c2} we get 
\begin{align*}  \int_{E\cap B_R(y)} \frac{(y-x)\cdot \nu(y)}{|x-y|^{d+s} } dx  & =\int_{H^-(y)\cap B_{R}(y) }   \frac{(y-x)\cdot \nu(y)}{|x-y|^{d+s} } dx-  \int_{(E\Delta H^-(y) )\cap B_{R}(y) }   \frac{|(y-x)\cdot \nu(y)|}{|x-y|^{d+s} } dx
 \\ \nonumber 
 &= \omega_{d-1}\frac{R^{1-s}}{1-s}-  \int_{(E\Delta H^-(y) )\cap B_{R}(y) }   \frac{|(y-x)\cdot \nu(y)|}{|x-y|^{d+s} } dx.   \end{align*}
Therefore, arguing as in Proposition \ref{pointwise}, we can show that  $\mathcal{H}(E)$ can be equivalently defined as follows, for all $R>0$
\begin{align}\label{leo} \mathcal{H}(E)= &- \omega_{d-1}\Per(E) (1+\log R)+ \int_{\partial^* E}\int_{(E\Delta H^-(y))\cap B_{R}(y) }  \frac{|(y-x)\cdot \nu(y)|}{|x-y|^{d+1} } dxd\mathcal{H}^{d-1}(y)\\ 
 & - \int_{\partial^* E} \int_{E\setminus B_R(y)} \frac{(y-x)\cdot \nu(y)}{|x-y|^{d+1} }  dxd\mathcal{H}^{d-1}(y).\nonumber \end{align} 
 This formula immediately gives the desired rescaling property \eqref{resc}. 
\end{proof} 

Now, we identify some classes of sets where $\mathcal{H}$ is bounded.  

\begin{proposition} \label{remreg}   Let $E$ be a  measurable set with $|E|<+\infty$ and $P(E)<+\infty$. 
\begin{enumerate} 
\item If $\partial E$ is uniformly of class  $C^{1, \alpha}$ for  some $\alpha>0$, then $\mathcal{H}(E)<+\infty$.
\item If $E$ is a convex set then, for every $s\in  (0,1)$, there holds 
\[\mathcal{H}(E)\leq \frac{(\dia E)^s
}{2} \int_{\partial^* E} H_s(E,y)d\mathcal{H}^{d-1}(y) -\omega_{d-1}\Per(E)\left(\frac{1}{s}+\log(\dia E)\right)\]
where  $\dia E:=\sup_{x,y\in E} |x-y|$, and $H_s(E,y)$ is the fractional mean curvature of $E$ at $y$, which is defined as 
\[H_s(E,y):=\int_{\R^d}\frac{\chi_{\R^d\setminus E}(x)-\chi_E(x)}{|x-y|^{d+s}}dx, \]
in the principal value sense. \end{enumerate}
\end{proposition}

 \begin{proof} 
 \begin{enumerate} \item  If  $\partial E$ is uniformly of class  $C^{1, \alpha}$, then  there exists $\eta>0$  such that for all $y\in \partial E$, $\partial E \cap B_\eta(y)$ is a graph of a $C^{1,\alpha}$ function  $h$, such that $\|\nabla h\|_{C^{0,\alpha}(B_\eta'(y))}\leq C$, for some $C$ independent of $y$. Up to a rotation and translation,  we may assume that $y=0$,  $h(0)=0$ and
  $\nabla h (0)=0$ and moreover  $-C|x'|^{1+\alpha} \leq h(x')\leq C|x'|^{1+\alpha} $ for all $x'\in B'_\eta$.
 Therefore recalling that   $E\cap B_\eta=\{(x, x_d) \ |\ x_d\leq h(x') \}$  and that $H^-(0)=\{(x',x_d)\ |\ x_d\leq 0\}$, there holds 
\[(E\Delta H^-(0))\cap B_\eta \subseteq C_\eta:=\{(x', x_d)\ | -C|x'|^{1+\alpha}\leq x_d\leq C|x'|^{1+\alpha}, |x'|\leq \eta\}.\]
We compute
\begin{align*} &\int_{(E\Delta H^-(0))\cap B_{1} }  \frac{1}{|x|^{d} } dx= \int_{(E\Delta H^-(0))\cap B_{\eta} }  \frac{1}{|x|^{d} } dx+ \int_{(E\Delta H^-(0))\cap (B_{1}\setminus B_\eta) }  \frac{1}{|x|^{d} } dx\\ & \leq \int_{C_\eta} \frac{1}{|x|^{d}}dx+ \frac{1}{2}\int_{ B_1\setminus B_\eta} \frac{1}{|x|^{d}}dx\leq   \int_{C_\eta} \frac{1}{|x'|^{d}}dx+ \frac{1}{2}\int_{ B_1\setminus B_\eta} \frac{1}{|x|^{d}}dx 
\\ &\leq 2C \int_{B'_\eta}\frac{|x'|^{1+\alpha}}{|x'|^{d}}dx' -\frac{1}{2} d\omega_d \log (\eta\wedge 1)=\frac{2C(d-1)\omega_{d-1} \eta^{\alpha}}{\alpha} -\frac{1}{2} d\omega_d \log (\eta\wedge 1).\end{align*} 
Then, recalling \eqref{h} we get that
\[\mathcal{H}(E)\leq \left(\frac{2C(d-1)\omega_{d-1} \eta^{\alpha}}{\alpha} -\frac{1}{2} d\omega_d \log (\eta\wedge 1)\right)\Per (E)
+ d\omega_d |E|<+\infty.\]
\item Let $R=\dia E$. Then by \eqref{leo}, we get 
\begin{align*} \mathcal{H}(E)&= - \omega_{d-1}\Per(E) (1+\log R) + \int_{\partial^* E}\int_{(E\Delta H^-(y))\cap B_{R}(y) }  \frac{|(y-x)\cdot \nu(y)|}{|x-y|^{d+1} } dxd\mathcal{H}^{d-1}(y)\\ & \leq 
 - \omega_{d-1}\Per(E) (1+\log R) + \int_{\partial^* E}\int_{(E\Delta H^-(y))\cap B_{R}(y) }  \frac{1}{|x-y|^{d} } dxd\mathcal{H}^{d-1}(y)
 \\ &\leq 
 - \omega_{d-1}\Per(E)  (1+\log R)+ \int_{\partial^* E}\int_{(E\Delta H^-(y))\cap B_{R}(y) }  \frac{R^s}{|x-y|^{d+s} } dxd\mathcal{H}^{d-1}(y).
 \end{align*} By convexity for every $y\in \partial^*E$, recalling that $E\subseteq B_R(y)$, there holds 
 \begin{align*}& \int_{(E\Delta H^-(y))\cap B_{R}(y) }  \frac{R^s}{|x-y|^{d+s} } dx =\frac{R^s}{2}\int_{B_R(y)} \frac{\chi_{\R^d\setminus E}(x)-\chi_E(x)}{|x-y|^{d+s}}dx\\ &=\frac{R^s}{2}H_s(E, y)-\frac{R^s}{2}\int_{\R^d\setminus B_R(y)}\frac{1}{|x-y|^{d+s}}dx=\frac{R^s}{2}H_s(E, y)-\frac{d\omega_d}{2s}.\end{align*} 
 Therefore, substituting this equality in the previous estimate, we get
 \[ \mathcal{H}(E)\leq \frac{R^s}{2} \int_{\partial^* E}H_s(E,y)d\mathcal{H}^{d-1}(y)- \omega_{d-1}\Per(E) (1+\log R)-\frac{d\omega_d}{2s}\Per(E). \]
 
 \end{enumerate} 
 \end{proof}
 \begin{remark}\upshape  Note that by  Proposition \ref{remreg}, $\mathcal{H}(Q)<+\infty$ for every cube $Q=\Pi_{i=1}^d [a_i, b_i]$. 
\\ Indeed for $y\in \partial^* Q$, there holds that $H_s(Q, y)\sim \frac{1}{(d(y, (\partial Q\setminus \partial^* Q)))^{s}}$    for $s\in (0,1)$   and so
 $\int_{\partial^* Q}H_s(Q,y)d\mathcal{H}^{d-1}(y)<+\infty$. %
\end{remark} 
 
Finally we provide some useful equivalent representations of the functional $\mathcal{H}$. 

\begin{proposition}\label{repr}
\ \ \  \begin{enumerate} 
\item[(i)]   Let  $E$ be a set with finite perimeter such that $ \mathcal{H}(E)<+\infty$. Then
\begin{align*}\mathcal{H}( E)=& -\frac{d\omega_{d-1}}{d-1}\Per(E)\\ &- \lim_{\delta\to 0^+}\Big[\frac{1}{d-1}  \int_{\partial^* E} \int_{\partial^* E\setminus B_\delta(y)} \frac{\nu(y) \cdot \nu(x) }{|x-y|^{d-1}} d\mathcal{H}^{d-1}(x)d\mathcal{H}^{d-1}(y)+\omega_{d-1}\log \delta \Per(E)\Big]. \end{align*} 
\item[(ii)] Let  $E$ be a 
compact set with boundary of class $C^2$. Then
\begin{align*}\mathcal{H}( E)=&   \frac{1}{d-1}  \int_{\partial E} \int_{\partial E }  \frac{(\nu(x)-\nu(y))^2}{2|x-y|^{d-1}}d\mathcal{H}^{d-1}(x)d\mathcal{H}^{d-1}(y)
-\frac{d\omega_{d-1}}{d-1}\Per(E)
\\& +\frac{1}{d-1}  \int_{\partial  E} \int_{\partial E} \frac{1}{|x-y|^{d-1}}\left|\frac{(y-x)}{|y-x|}\cdot \nu(x)\right|^2((d-1)\log|x-y|-1) d\mathcal{H}^{d-1}(x)d\mathcal{H}^{d-1}(y) \\
&+ \int_{\partial E} \int_{\partial E } \frac{H(E,x)\nu(x)\cdot (y-x)}{|y-x|^{d-1}} \log|x-y| d\mathcal{H}^{d-1}(x)d\mathcal{H}^{d-1}(y).
\end{align*}
  \end{enumerate}
\end{proposition} 

\begin{proof}
\begin{enumerate} 
\item[(i)] If the diameter of $E$ is less than $1$, then $E\setminus B_1(y)=\emptyset$ for all $y\in \partial E$, and so
\begin{align*} \mathcal{H}( E)& = -\omega_{d-1}\Per(E)+ \int_{\partial^* E}\int_{(E\Delta H^-(y))\cap B_{1}(y) }  \frac{|(y-x)\cdot \nu(y)|}{|x-y|^{d+1} } dxd\mathcal{H}^{d-1}(y).\end{align*}  

Using that  \[\frac{1}{d-1}\text{div}_x\left(\frac{\nu(y)}{|x-y|^{d-1}}\right)=\frac{ (y-x)\cdot \nu(y)}{|x-y|^{d+1}}\]  we compute the second inner integral for $y\in \partial^* E$, recalling that $E\subset B_1(y)$, \begin{align*} 
&\int_{(E\Delta H^-(y))\cap B_{1}(y)   }  \frac{|(y-x)\cdot \nu(y)|}{|x-y|^{d+1} } dx\\ =&\int_{( H^-(y)\setminus E) \cap B_{1}(y)  }  \frac{(y-x)\cdot \nu(y)}{|x-y|^{d+1} } dx-\int_{(E\setminus H^-(y))   }  \frac{(y-x)\cdot \nu(y)}{|x-y|^{d+1} } dx
\\
=&\lim_{\delta\to 0} \Big[  \int_{ (H^-(y)\setminus E) \cap (B_{1}(y)\setminus B_\delta(y)) }  \frac{(y-x)\cdot \nu(y)}{|x-y|^{d+1} } dx-\int_{(E\setminus H^-(y)) \setminus B_\delta(y)  }  \frac{(y-x)\cdot \nu(y)}{|x-y|^{d+1} } dx\Big]
\\
=&\lim_{\delta\to 0} \Big[ -\frac{1}{d-1}\int_{\partial^* E\setminus B_\delta(y) } \frac{\nu(x)\cdot \nu(y)}{|x-y|^{d-1}}d\mathcal{H}^{d-1}(x) +\frac{1}{d-1}\int_{\partial B_1(y)\cap H^-(y)}\nu(x)\cdot \nu(y)d\mathcal{H}^{d-1}(x)\\ 
&+ \frac{1}{d-1}  \int_{\partial H^-(y)\cap (B_1(y)\setminus B_\delta(y)) } \frac{1}{|x-y|^{d-1}}d\mathcal{H}^{d-1}(x)- \frac{1}{\delta^{d-1}} \int_{\partial B_\delta(y)\cap (H^-(y)\Delta E)} \nu(x)\cdot \nu(y) d\mathcal{H}^{d-1}(x)\Big].\end{align*} 
Now we observe that 
\begin{align*}&\lim_{\delta\to 0} \frac{1}{\delta^{d-1}} \int_{\partial B_\delta(y)\cap (H^-(y)\Delta E)} |\nu(x)\cdot \nu(y)|d\mathcal{H}^{d-1}(x)\leq\lim_{\delta\to 0} \frac{1}{\delta^{d-1}} \int_{\partial B_\delta(y)\cap (H^-(y)\Delta E)} d\mathcal{H}^{d-1}(x)\\ & = \lim_{\delta\to 0} \int_{\partial B_1\cap\left( H^-(y)\Delta \frac{(E-y)}{\delta}\right)} d\mathcal{H}^{d-1}(x)=0 \end{align*}
since, for $y\in\partial^* E$, there holds that $\frac{(E-y)}{\delta}\to H^-(y)$ locally in $L^1$ as $\delta\to 0$, see \cite[Thm II.4.5]{maggibook}. 
We compute 
\[\frac{1}{d-1}\int_{\partial B_1(y)\cap H^-(y)}\nu(x)\cdot \nu(y)d\mathcal{H}^{d-1}(x)=\frac{1}{d-1}\int_{x_d=-\sqrt{1-|x'|^2}}x_d d\mathcal{H}^{d-1}(x)=-\frac{\omega_{d-1}}{d-1} \]
and 
\begin{equation*}\label{contoh}   \frac{1}{d-1}\int_{\partial H^-(y)\cap (B_1(y)\setminus B_\delta(y)) } \frac{1}{|x-y|^{d-1}}d\mathcal{H}^{d-1}(x)= \frac{1}{d-1}\int_{B_1'\setminus B_\delta'} \frac{1}{|x'|^{d-1}}dx'=- \omega_{d-1} \log \delta.\end{equation*}

Therefore  
\begin{align*}\mathcal{H}( E)=& -\frac{d \omega_{d-1}}{d-1}\Per(E)\\ &- \lim_{\delta\to 0^+}\Big[\frac{1}{d-1}  \int_{\partial^* E} \int_{\partial^* E\setminus B_\delta(y)} \frac{\nu(y) \cdot \nu(x) }{|x-y|^{d-1}} d\mathcal{H}^{d-1}(x)d\mathcal{H}^{d-1}(y)+\omega_{d-1}\log \delta \Per(E)\Big]. \end{align*}
If $\partial E$ has diameter greater or equal to $1$, we obtain the formula by rescaling, using \eqref{resc}.

\item[(ii)]  Let us fix $y\in \partial E$ and define   for all $x\in \partial E$, $x\neq y$,  the vector  field
\[\eta(x)= f(|x-y|)(y-x)\qquad \text{ where  $f(r):= \frac{\log r}{r^{d-1}}.$} \]  
By the Gauss-Green Formula (see \cite[I.11.8]{maggibook}), for $\delta>0$ there holds  
\begin{align*} &\frac{1}{d-1}\int_{\partial E\setminus B_\delta(y)} \div_{\tau} \eta(x)d\mathcal{H}^{d-1}(x)\\ & =\int_{\partial E\setminus B_\delta(y)} H(E,x)\nu(x)\cdot \eta(x)d\mathcal{H}^{d-1}(x)+\frac{1}{d-1}\int_{\partial B_\delta(y)\cap \partial E} \eta(x)\cdot \frac{x-y}{|x-y|} d\mathcal{H}^{d-2}(x)\\
& = \int_{\partial E\setminus B_\delta(y)} H(E,x)\nu(x)\cdot \eta(x)d\mathcal{H}^{d-1}(x)-  \omega_{d-1}\log \delta \end{align*} 
where $\div_\tau \eta(x)$ is the tangential divergence, that is $\div_\tau\eta(x)=\div\eta(x)- \nu(x)^T\nabla \eta(x)\nu(x)$.
Therefore  integrating the previous equality on  $ \partial E$,  we get  that 
\begin{align}\label{gg}    \omega_{d-1}\log \delta \Per(E)= & \int_{\partial E} \int_{\partial E\setminus B_\delta(y)} H(E,x)\nu(x)\cdot \eta(x)d\mathcal{H}^{d-1}(x)d\mathcal{H}^{d-1}(y)\\ &-
\frac{1}{d-1} \int_{\partial E}  \int_{\partial E\setminus B_\delta(y)} \div_{\tau} \eta(x)d\mathcal{H}^{d-1}(x)d\mathcal{H}^{d-1}(y).\nonumber \end{align} 
Now we compute 
\begin{align*} &\div_{\tau} \eta(x) = \tr\nabla \eta(x)- \nu(x)^T\nabla \eta(x)\nu(x)\\ 
&= -\tr \left(f(|x-y|)\mathbf{I} +f'(|x-y|)|x-y| \frac{y-x}{|x-y|}\otimes  \frac{y-x}{|x-y|}\right)\\
&+ \nu(x)^T\left(f(|x-y|)\mathbf{I} +f'(|x-y|)|x-y| \frac{y-x}{|x-y|}\otimes  \frac{y-x}{|x-y|}\right)\nu(x)\\
&= -f(|x-y|)d - f'(|x-y|)|x-y| +f(|x-y|)+f'(|x-y|)|x-y| \left|\frac{y-x}{|y-x|}\cdot \nu(x)\right|^2\\
&= -\frac{1}{|x-y|^{d-1}}+\frac{1-(d-1)\log |x-y|}{|x-y|^{d-1})}\left|\frac{y-x}{|y-x|}\cdot \nu(x)\right|^2 
\end{align*} 
where we used the equality $rf'(r)= \frac{1}{r^{d-1}}-(d-1)f(r)=\frac{1-(d-1)\log r}{r^{d-1}}$.\\
If we substitute this expression in \eqref{gg} we get 
\begin{align*}   &   \omega_{d-1}\log \delta \Per(E)= \int_{\partial E} \int_{\partial E\setminus B_\delta(y)} \frac{H(E,x)\nu(x)\cdot (y-x)}{|x-y|^{d-1}} \log|x-y| d\mathcal{H}^{d-1}(x)d\mathcal{H}^{d-1}(y)\\ &+
\frac{1}{d-1} \int_{\partial E}  \int_{\partial E\setminus B_\delta(y)}\frac{1}{|x-y|^{d-1}}d\mathcal{H}^{d-1}(x)d\mathcal{H}^{d-1}(y) \\
&-\frac{1}{d-1}  \int_{\partial E}  \int_{\partial E\setminus B_\delta(y)}\frac{1-(d-1)\log |x-y|}{|x-y|^{d-1})}\left|\frac{y-x}{|y-x|}\cdot \nu(x)\right|^2d\mathcal{H}^{d-1}(x)d\mathcal{H}^{d-1}(y).
\end{align*} 
The conclusion then follows by substituting $ \omega_{d-1}\log \delta \Per(E)$ with the previous expression in  the representation formula obtained in (i), and observing that $1-\nu(x)\nu(y)= (\nu(x)-\nu(y))^2/2$. 
\end{enumerate} 
 \end{proof}

\end{document}